\documentclass[12pt]{article}

\usepackage{amsmath}
\usepackage{amsfonts}
\usepackage{bigints}
\usepackage{amsthm}

\begin{document}

\newcommand{\A}{\mbox{${{{\cal A}}}$}}


\author{Attila Losonczi}
\title{Means of infinite sets III}

\date{25 Sept 2018}

\newtheorem{thm}{\qquad Theorem}[section]
\newtheorem{prp}[thm]{\qquad Proposition}
\newtheorem{lem}[thm]{\qquad Lemma}
\newtheorem{cor}[thm]{\qquad Corollary}
\newtheorem{rem}[thm]{\qquad Remark}
\newtheorem{ex}[thm]{\qquad Example}
\newtheorem{prb}[thm]{\qquad Problem}
\newtheorem{df}[thm]{\qquad Definition}

\maketitle

\begin{abstract}

\noindent

We study various topics, e.g. accumulation points by a mean, two types of derivative by a mean, two new continuity and a boundedness concepts, we construct new means from old ones, finally we investigate the limit of means.

\noindent
\footnotetext{\noindent
AMS (2010) Subject Classifications: 26E60, 28A10 \\

Key Words and Phrases: generalized means of sets, Lebesgue and Hausdorff measure}

\end{abstract}

\section{Introduction}
In this paper we are going to continue the investigations started in \cite{lamis} and \cite{lamisii}. For basic definitions, examples, ideas, intentions please consult \cite{lamis} and  \cite{lamisii}. More on this area can be found in \cite{lamubs}, \cite{lambm} and \cite{lamsbm}.

In this paper we study various topics. First we generalize the liminf, limsup of sets in a way that we identify the maximum irrelevant part of the set by the mean. Using that we extend the notion of internality. Then we define accumulation point by a mean such that each of its neighborhood contains an essential part of the set i.e. it contains a subset which leaving out spoils the mean.
We investigate its properties, relations to previous notions and we can also define closed sets by a mean. For $Avg^1$ we show that the closed sets constitute a topology but it fails to be valid for ${\cal{M}}^{acc}$.

Our next aim is to find a reasonably good interpretation of the derivative by a mean. We found two different ways. The first one measures that how symmetric the set is around $x$ in the sense of the mean. The second is about the lower and upper derivative for compacts sets using the Hausdorff metric. We study $Avg^1$ in details and find necessary and sufficient condition when they take the extremum. 

Then we investigate new concepts on continuity and boundedness. The former is related to Cantor-continuity while the latter is a kind of triangle-inequality for sets when we are dealing with generalized means. We show that means by measures satisfy both conditions.

Then we construct new means from old ones that resembles to the way how one defines a quasi-arithmetic mean from the arithmetic mean. We present attributes that are inherited by this method.

Finally we study the limit of means. In \cite{lamis} we defined many generalized means that were constructed via (pointwise) limit of means. First we analyze the underlying means and e.g. we show that $H\mapsto Avg(S(H,\delta))$ is continuous regarding the Hausdorff metric on compact sets. Then we investigate the general limits and define pointwise and uniform convergence. We present many properties that are inherited by pointwise limits. We also prove that $LAvg$ was not derived by uniform convergence.

At the end of the paper we present a few open problems. 

\subsection{Basic notations}
For easier readability we copy the basic notations from \cite{lamis}.

Throughout this paper function $\A()$ will denote the arithmetic mean of any number of variables. Moreover if $(a_n)$ is an infinite sequence and $\lim_{n\to\infty}\A(a_1,\dots,a_n)$ exists then $\A((a_n))$ will denote its limit.

\begin{df}\label{davg}Let $\mu^s$ denote the s-dimensional Hausdorff measure ($0\leq s\leq 1$). If $0<\mu^s(H)<+\infty$ (i.e. $H$ is an $s$-set) and $H$ is $\mu^s$ measurable then $$Avg(H)=\frac{\int\limits_H x\ d\mu^s}{\mu^s(H)}.$$
If $0\leq s\leq 1$ then set $Avg^s=Avg|_{\{\text{measurable s-sets}\}}$. E.g. $Avg^1$ is $Avg$ on all Lebesgue measurable sets with positive measure.
\end{df}

\smallskip

For $K\subset\mathbb{R},\ y\in\mathbb{R}$ let us use the notation $$K^{y-}=K\cap(-\infty,y],K^{y+}=K\cap[y,+\infty).$$

\smallskip

If $H\subset\mathbb{R},\ \epsilon>0$ we use the notation $S(H,\epsilon)=\bigcup_{x\in H}S(x,\epsilon)$ where $S(x,\epsilon)=\{y:|x-y|<\epsilon\}$. 

\smallskip

Let ${\cal{I}}$ be an ideal of subsets of $\mathbb{R}$ and $H\subset\mathbb{R},\ H\notin{\cal{I}}$ be bounded. Set $\varlimsup^{\cal{I}} H=\inf\{x: H^{x+}\in{\cal{I}}\}$. Similarly $\varliminf^{\cal{I}} H=\sup\{x:H^{x-}\in{\cal{I}}\}$.

\medskip

Let $T_s$ denote the reflection to point $s\in\mathbb{R}$ that is $T_s(x)=2s-x\ (x\in\mathbb{R})$.

If $H\subset\mathbb{R},x\in\mathbb{R}$ then set $H+x=\{h+x:h\in H\}$. Similarly $\alpha H=\{\alpha h:h\in H\}\ (\alpha\in\mathbb{R})$.

We use the convention that these operations $+,\cdot$ have to be applied prior to the set theoretical operations, e.g. $H\cup K\cup L+x=H\cup K\cup (L+x)$.

\smallskip

$int(H),cl(H),mar(H), H'$ will denote the interior, closure, boundary and accumulation points of $H\subset\mathbb{R}$ respectively. Let $\varliminf H=\inf H',\ \varlimsup H=\sup H'$ for infinite bounded $H$.

\smallskip

Usually ${\cal{K}},{\cal{M}}$ will denote means, $Dom({\cal{K}})$ denotes the domain of ${\cal{K}}$. When we simply refer to ${\cal{K}}(H)$ then we automatically mean that $H\in Dom\ {\cal{K}}$ without mentioning that before.

\section{New concepts}

\subsection{Some new properties}

\begin{df}${\cal{K}}$ is \textbf{equi-monotone} if $H_1\cap H_2=\emptyset,{\cal{K}}(H_1\cup H_2)={\cal{K}}(H_1)$ implies that ${\cal{K}}(H_1)={\cal{K}}(H_2)$.
\end{df}

\begin{prp}$Avg$ is not equi-monotone, however $Avg^s$ is equi-monotone ($0\leq s\leq 1$).
\end{prp}
\begin{proof}Let $H_1,H_2$ be an $0.7$ and $0.5$-sets respectively that are disjoint and have different $Avg$. They clearly show that $Avg$ is not equi-monotone.

To prove that $Avg^s$ is equi-monotone let $H_1,H_2$ be disjoint $s$-sets. Then rearranging the equation
\[Avg^s(H_1\cup H_2)=\frac{\mu^s(H_1)Avg^s(H_1)+\mu^s(H_2)Avg^s(H_2)}{\mu^s(H_1)+\mu^s(H_2)}=Avg^s(H_1)\]
gives the statement.
\end{proof}

This gives that disjoint-monotonicity does not imply equi-monotonicity since $Avg$ is dijoint-monotone.

\begin{prp}${\cal{M}}^{acc}$ is not equi-monotone. If ${\cal{M}}^{acc}$ is restricted to sets that have the same level, then equi-monotonicity holds for those sets.
\end{prp}
\begin{proof}Let $H_1=\{\frac{1}{n}:n\in\mathbb{N}\},\ H_2=\{2\}$. These show that ${\cal{M}}^{acc}$ is not equi-monotone.

Let $lev(H_1)=lev(H_2)=n$. Then ${\cal{M}}^{acc}(H_i)=\A(H_i^{(n)})\ (i=1,2)$ and $(H_1\cup H_2)^{(n)}=H_1^{(n)}\cup H_2^{(n)}$. Hence the previous proposition gives the assertion because $\A=Avg^0$.
\end{proof}

\begin{prp}If ${\cal{K}}$ is monotone, $H\in Dom({\cal{K}})$, $x,y,z\in\mathbb{R},\ x<y<z,\ {\cal{K}}(H^{x+})={\cal{K}}(H^{z+})=k$ then ${\cal{K}}(H^{y+})=k$. A similar statement is true for $-$ instead of $+$.
\end{prp}
\begin{proof}Clearly $\sup H\cap[y,z)\leq z\leq\inf H^{z+}$ hence ${\cal{K}}(H^{y+})={\cal{K}}\big((H\cap[y,z))\cup H^{z+}\big)\leq{\cal{K}}(H^{z+})$. Similarly $\sup H\cap[x,y)\leq\inf H^{y+}$ hence ${\cal{K}}(H^{x+})={\cal{K}}((H\cap[x,y))\cup H^{y+})\leq{\cal{K}}(H^{y+})$.
\end{proof}

\begin{df}Let ${\cal{K}}$ be a monotone mean, $H\in Dom({\cal{K}})$. Let

$\varliminf_{\cal{K}}H=\sup\{x:{\cal{K}}(H)={\cal{K}}(H^{x+})\}$, 
$\varlimsup_{\cal{K}}H=\inf\{x:{\cal{K}}(H)={\cal{K}}(H^{x-})\}$ be the liminf and limsup of $H$ with respect to the mean ${\cal{K}}$.
\end{df}

Actually monotonicity is not necessary i.e. we could have formulated the definition without that.

\begin{prp}If ${\cal{K}}$ is a mean and $H\in Dom({\cal{K}})$ then $\varliminf_{\cal{K}}H\leq{\cal{K}}(H)\leq\varlimsup_{\cal{K}}H$.
\end{prp}
\begin{proof}We prove the first inequality (the other is similar). Suppose indirectly that $\varliminf_{\cal{K}}H>{\cal{K}}(H)$. Then there is an $x\in({\cal{K}}(H),\varliminf_{\cal{K}}H)$ for which ${\cal{K}}(H)={\cal{K}}(H^{x+})$. But ${\cal{K}}(H)<x$ implies that ${\cal{K}}(H)<{\cal{K}}(H^{x+})$ which is a contradiction.
\end{proof}

\begin{df}${\cal{K}}$ is strong internal with respect to itself if $\varliminf_{\cal{K}}H\leq{\cal{K}}(H)\leq\varlimsup_{\cal{K}}H$. 
\ ${\cal{K}}$ is strict strong internal with respect to itself if it is strong internal and $\varliminf_{\cal{K}}H<{\cal{K}}(H)<\varlimsup_{\cal{K}}H$ whenever $\varliminf_{\cal{K}}H\ne\varlimsup_{\cal{K}}H$.
\end{df}

\begin{prp}$\varliminf_{Avg^1}H=\varliminf_{{\cal{N}}_0}H,\ \varlimsup_{Avg^1}H=\varlimsup_{{\cal{N}}_0}H$ where ${\cal{N}}_0=\{$sets with Lebesgue measure 0$\}$ and $\varliminf_{{\cal{N}}_0},\varlimsup_{{\cal{N}}_0}$ denote the liminf, limsup by ideal ${\cal{N}}_0$.
\end{prp}
\begin{proof}If $\lambda(H^{x-})=0$ then $Avg^1(H^{x+})=Avg^1(H)$ which shows that  $\varliminf_{{\cal{N}}_0}H\leq\varliminf_{Avg^1}H$. Let us assume that there is an $x\in (\varliminf_{{\cal{N}}_0}H,\varliminf_{Avg^1}H)$ such that $Avg^1(H^{x+})=Avg^1(H)$. Then by
$$\frac{\lambda(H^{x-})Avg^1(H^{x-})+\lambda(H^{x+})Avg^1(H^{x+})}{\lambda(H)}=Avg^1(H)$$
rearranging the equation and using that $\lambda(H^{x-})>0$ we get that $Avg^1(H^{x-})=Avg^1(H)$. That gives that $Avg^1(H)=x$ which is a contradiction.

$\varlimsup$ can be handled similarly.
\end{proof}

\begin{cor}\label{cavg1ssii}$Avg^1$ strict strong internal with respect to itself .
\end{cor}
\begin{proof}In \cite{lambm} Proposition 2.7. we actually proved that $\varliminf_{{\cal{N}}_0}H<Avg^1(H)<\varlimsup_{{\cal{N}}_0}H$.
\end{proof}

Normally we cannot expect that ${\cal{K}}([\varliminf_{\cal{K}}(H),\varlimsup_{\cal{K}}(H)]\cap H)={\cal{K}}(H)$ holds. E.g. it can happen that $[\varliminf_{\cal{K}}(H),\varlimsup_{\cal{K}}(H)]\cap H=\emptyset$ as a simple example can show for ${\cal{M}}^{iso}$ (let $H=\{-\frac{1}{n},1+\frac{1}{n}:n\in\mathbb{N}\}$). However under the below conditions it holds.

\begin{prp}\label{pkhcblils}Let ${\cal{K}}$ be slice-continuous. Then ${\cal{K}}(H)={\cal{K}}(H\cap[\varliminf_{\cal{K}} H,\varlimsup_{\cal{K}} H])$.
\end{prp}
\begin{proof}Let us observe that if $\varliminf_{\cal{K}} H=\varlimsup_{\cal{K}} H$ then the statement obviously holds. Now suppose that $\varliminf_{\cal{K}} H\ne\varlimsup_{\cal{K}} H$.

We know that $g(x)={\cal{K}}(H^{x+})$ is continuous. If $x<\varliminf_{\cal{K}} H$ then $g(x)={\cal{K}}(H)$. Hence $g(\varliminf_{\cal{K}} H)={\cal{K}}(H)$ but $g(\varliminf_{\cal{K}} H)={\cal{K}}(H\cap[\varliminf_{\cal{K}} H,+\infty))$. 

Let $H_1=H\cap[\varliminf_{\cal{K}} H,+\infty)$. Applying similar argument for $\varlimsup_{\cal{K}} H_1=\varlimsup_{\cal{K}} H$ we get that ${\cal{K}}(H)={\cal{K}}(H_1)={\cal{K}}(H_1\cap(-\infty,\varlimsup_{\cal{K}} H_1])={\cal{K}}(H\cap[\varliminf_{\cal{K}} H,\varlimsup_{\cal{K}} H])$.
\end{proof}

\subsection{Accumulation points by a mean}

An accumulation point is a kind of point for which each of its neighborhoods contains an essential part of the set. We transplant this notion for means.

\begin{df}Let ${\cal{K}}$ be a mean, $H\in Dom({\cal{K}})$. Set 
$$H^{'{\cal{K}}}=\{x\in\mathbb{R}:\forall\epsilon>0\ \exists L\subset S(x,\epsilon)\text{ such that }{\cal{K}}(H-L)\ne{\cal{K}}(H)\}.$$
We call $H^{'{\cal{K}}}$ the accumulation points of $H$ by ${\cal{K}}$.
\end{df}

\begin{df}
${\cal{K}}$ is said to be self-accumulated if ${\cal{K}}(H^{'{\cal{K}}})={\cal{K}}(H)$.
\end{df}

\begin{prp}If ${\cal{K}}$ is finite-independent then $H^{'{\cal{K}}}\subset H'$. 
\end{prp}
\begin{proof}If $x\notin H'$ then there is $\epsilon>0$ such that $S(x,\epsilon)$ contains at most one point from $H$.
\end{proof}

\begin{prp}$H^{'{\cal{K}}}$ is closed.
\end{prp}
\begin{proof}Let $x\in\mathbb{R}$ such that $\forall\epsilon>0\ \exists y\in S(x,\epsilon)$ such that $y\in H^{'{\cal{K}}}$. Then $\exists\delta>0$ such that $S(y,\delta)\subset S(x,\epsilon)$ and $\exists L\subset S(y,\delta)\text{ such that }{\cal{K}}(H-L)\ne{\cal{K}}(H)$. But then $L\subset S(x,\epsilon)$ showing that $x\in H^{'{\cal{K}}}$.
\end{proof}

\begin{prp}If ${\cal{K}}$ is union-monotone and slice-continuous then 

$\varliminf_{\cal{K}}(H)=\min H^{'{\cal{K}}},\ \varlimsup_{\cal{K}}(H)=\max H^{'{\cal{K}}}$.
\end{prp}
\begin{proof}We show it for the $\min$, the $\max$ is similar. Let $x=\varliminf_{\cal{K}}(H)$. Then $\forall\epsilon>0\ {\cal{K}}(H)\ne{\cal{K}}(H^{(x+\epsilon)+})$ that means that ${\cal{K}}(H)\ne{\cal{K}}(H-(x,x+\epsilon))$. Because otherwise ${\cal{K}}(H)={\cal{K}}(H-(x,x+\epsilon))$ and by \ref{pkhcblils} ${\cal{K}}(H)={\cal{K}}(H-(-\infty,x])$ would imply by \cite{lamisii}2.26(2) that ${\cal{K}}(H)={\cal{K}}(H-(-\infty,x+\epsilon))$ that is a contradiction. This gives that $x\in H^{'{\cal{K}}}$.

Let $x\in H^{'{\cal{K}}}$. Suppose that $x<\varliminf_{\cal{K}}(H)$. Let $\epsilon>0$ such that $x+\epsilon<\varliminf_{\cal{K}}(H)$. We know that $\exists  L\subset S(x,\epsilon)\text{ such that }{\cal{K}}(H-L)\ne{\cal{K}}(H)$. We can write ${\cal{K}}((H^{x+}-L)\cup (H^{x-}-L))={\cal{K}}(H-L)\ne{\cal{K}}(H)$. Then $x<\varliminf_{\cal{K}}(H)$ implies that  ${\cal{K}}((H^{x+}-L)\cup L)={\cal{K}}(H^{x+})={\cal{K}}(H)$. By union-monotonicity (and using that either ${\cal{K}}(H-L)<{\cal{K}}(H)$ or ${\cal{K}}(H-L)>{\cal{K}}(H)$) we get that ${\cal{K}}((H^{x+}-L)\cup (H^{x-}-L)\cup L)={\cal{K}}(H)\ne{\cal{K}}(H)$ which is a contradiction.
\end{proof}

\begin{lem}\label{lflfh}Let $H,L\subset\mathbb{R}$ be bounded Lebesgue measurable, $\lambda(H)>0,\ \lambda(L)>0$, $L\subset H$. If $Avg(H)\ne Avg(L)$ then $Avg(H)\ne Avg(H-L)$.
\end{lem}
\begin{proof}Assume the contrary. Rearranging the equation
\[Avg(H)=\frac{\lambda(H-L)Avg(H-L)+\lambda(L)Avg(L)}{\lambda(H)}\]
we would get that $Avg(H)=Avg(L)$ - a contradiction. 
\end{proof}

\begin{prp}\label{phpke}Let ${\cal{K}}=Avg^1,H\subset\mathbb{R}$ be bounded Lebesgue measurable, $\lambda(H)>0$. Then $H^{'{\cal{K}}}=\{x\in\mathbb{R}:\forall\delta>0\ \lambda(H\cap S(x,\delta))>0\}$.
\end{prp}
\begin{proof}If $\lambda(L)=0$ then clearly ${\cal{K}}(H-L)={\cal{K}}(H)$ which implies that if $x\in H^{'{\cal{K}}}$ then $\forall\delta>0\ \lambda(H\cap S(x,\delta))>0$.

Let us assume that $x\in\mathbb{R}$ such that $\forall\delta>0\ \lambda(H\cap S(x,\delta))>0$ holds. Let $\epsilon>0$. By \ref{lflfh} it is enough to find an $L\subset H\cap S(x,\epsilon),\ \lambda(L)>0$ such that $Avg(H)\ne Avg(L)$. If $x\ne Avg(H)$ then it is trivial. If they are equal then either $\lambda(H\cap (x,x+\epsilon))>0$ or $\lambda(H\cap (x-\epsilon,x))>0$. Choose the one with positive measure for $L$. ${\cal{K}}$ is strict strong internal hence ${\cal{K}}(L)\ne x$.
\end{proof}

\begin{cor}Let ${\cal{K}}=Avg^1,H\subset\mathbb{R},\lambda(H)>0$. If the lower Lebesgue density of $x\in\mathbb{R}$ regarding $H$ is greater than 0 then $x\in H^{'{\cal{K}}}$. Hence almost every point of $H$ is in $H^{'{\cal{K}}}$.\qed
\end{cor}

\begin{lem}\label{lavgmin}Let $a<b$ and $h<b-a$ be given. Then 
\[\min\big\{Avg^1(H): H\subset[a,b]\text{ is Lebesgue measurable, }\lambda(H)=h\big\}=a+\frac{h}{2},\]
\[\max\big\{Avg^1(H): H\subset[a,b]\text{ is Lebesgue measurable, }\lambda(H)=h\big\}=b-\frac{h}{2}.\]
\end{lem}
\begin{proof}We prove the statement for $\min$, the other is similar. 

Clearly $Avg^1([a,a+h])=a+\frac{h}{2}$. 

We show that $\int\limits_{H}xd\lambda\geq\int\limits_{[a,a+h]}xd\lambda$ holds and that will give the statement.

Let us take similar step functions in the following way. Let $n\in\mathbb{N}$. Set $f_n(x)=a+\frac{k}{n}h$ if $x\in[a+\frac{k}{n}h,a+\frac{k+1}{n}h]\ (k\in\{0,\dots,n-1\})$. Clearly for $H$ and $k\in\{0,\dots,n-1\}$ there is $x_k\in[a,b]$ such that $\lambda(H\cap[x_k,x_{k+1}])=\frac{1}{n}h$. Now set $g_n(x)=x_k$ if $x\in H\cap [x_k,x_{k+1})$. Obviously $f_n\to x$ on $[a,a+h]$ and $g_n\to x$ on $H$. Hence $\int\limits_{[a,a+h]}f_nd\lambda\to\int\limits_{[a,a+h]}xd\lambda$ and $\int\limits_{H}g_nd\lambda\to\int\limits_{H}xd\lambda$. But $a+\frac{k}{n}h\leq x_k$ implies that $\int\limits_{[a,a+h]}f_nd\lambda\leq\int\limits_{H}g_nd\lambda$ for every $n\in\mathbb{N}$ which gives the statement.
\end{proof}

\begin{prp}$Avg^1$ is not self-accumulated.
\end{prp}
\begin{proof}Let $K$ be a Lebesgue measurable set in $[0,1]$ that is dense such that $\forall x\in[0,1]\ \forall\epsilon>0\ \lambda(K\cap S(x,\epsilon))>0$. Moreover let $\lambda(K)=\frac{1}{2}$. Evidently such set exists.

Then let $H=K\cup [1,2]$. By \ref{phpke} $Avg^1(H^{'Avg^1})=Avg^1([0,2])=1$. By \ref{lavgmin} we get that  
\[Avg^1(H)=\frac{0.5Avg^1(K)+Avg^1([1,2])}{1.5}\geq\frac{0.5Avg^1([0,0.5])+Avg^1([1,2])}{1.5}>1\]
showing that $Avg^1$ is not self-accumulated.
\end{proof}

\begin{lem}\label{lamap}Let $H\subset\mathbb{R}$ be finite. Then $H^{'\A}=H-\{\A(H)\}$.
\end{lem}
\begin{proof}Let $|H|=n$. Clearly $H^{'\A}\subset H$. If $k\in H$ then \[\A(H)=\frac{\sum\limits_{h\in H}h}{n}\ne\frac{\sum\limits_{h\in H,h\ne k}h}{n-1}=\A(H-\{k\})\]
if and only if $k\ne\A(H-\{k\})$ if and only if $k\ne\A(H)$. This gives the statement.
\end{proof}

\begin{prp}Let ${\cal{K}}={\cal{M}}^{acc}$, $H\in Dom({\cal{K}})$ and $lev(H)=n\in\mathbb{N}$. Then $H^{'{\cal{K}}}=H^{(n)}-\{{\cal{K}}(H)\}$.
\end{prp}
\begin{proof}It is a straightforward consequence of \ref{lamap}.
\end{proof}

\begin{cor}${\cal{M}}^{acc}$ is self-accumulated.\qed
\end{cor}

We need a similar lemma than in \cite{lamis} 3.2.

\begin{lem}\label{ledsm2}Let $(H_n),(L_n)$ be two infinite sequences of finite sets such that all sets are uniformly bounded, $\forall n\ L_n\subset H_n$ and $\A(H_n)\to a$. Moreover $\lim_{n\to\infty}\frac{|L_n|}{|H_n|}=0$. Then $\A(H_n - L_n)\to a$.
\end{lem}
\begin{proof}Clearly 
\[\A(H_n)=\frac{|H_n-L_n|}{|H_n|}\A(H_n-L_n)+\frac{|L_n|}{|H_n|}\A(L_n).\] 
$\A(L_n)$ is bounded, $\frac{|L_n|}{|H_n|}\to 0$ and $\frac{|H_n-L_n|}{|H_n|}\to 1$ give the statement.
\end{proof}

\begin{ex}Let ${\cal{K}}={\cal{M}}^{iso}$. We construct a set $H$ for which $H'$ is infinite while $H^{'{\cal{K}}}=\emptyset$.
\end{ex}
\begin{proof}Let us take the Cantor set $C$. Its complement $[0,1]-C$ can be written in the form $\cup_{i=1}^{\infty}I_i$ where $(I_i)$ are the usual open disjoint intervals. Now for each end point of each interval add convergent sequence that converge to that point and remains in the interval. Moreover do it in the way that let the added sequences be symmetric to $\frac{1}{2}$. Let $H$ be union of $C$ and the points of the added sequences. Clearly $H'=C$ and by symmetry ${\cal{M}}^{iso}(H)$ exists and equals to $\frac{1}{2}$. If $x\in C$ then even if we leave out a whole $\epsilon$ neighbourhood of $x$ from $H$, it would not affect the mean by \ref{ledsm2}.
\end{proof}

\begin{df}$H\subset\mathbb{R}$ is called ${\cal{K}}$-closed if $H^{'{\cal{K}}}\subset H$.
\end{df}

\begin{prp}Let ${\cal{K}}=Avg^1, H_1,H_2,H_i\in Dom\ {\cal{K}}\ (i\in I)$. Then $(H_1\cup H_2)^{'{\cal{K}}}=H_1^{'{\cal{K}}}\cup H_2^{'{\cal{K}}}$ and $(\bigcap\limits_{i\in I}H_i)^{'{\cal{K}}}\subset \bigcap\limits_{i\in I}H_i^{'{\cal{K}}}$.
\end{prp}
\begin{proof}First let us note that if $A\subset B$ then $A^{'{\cal{K}}}\subset B^{'{\cal{K}}}$. This gives the second statement and $H_1^{'{\cal{K}}}\cup H_2^{'{\cal{K}}}\subset (H_1\cup H_2)^{'{\cal{K}}}$. For the second observe that $\lambda(S(x,\epsilon)\cap(H_1\cup H_2))>0$ implies that either $\lambda(S(x,\epsilon)\cap H_1)>0$  or $\lambda(S(x,\epsilon)\cap H_2)>0$ holds. Now if we take $\epsilon=\frac{1}{n}$ for all $n\in\mathbb{N}$ then there is $i\in\{1,2\}$ such that $\lambda(S(x,\epsilon)\cap H_i)>0$ holds for infinitely many $n$. Hence $x\in H_i^{'{\cal{K}}}$.
\end{proof}

\begin{cor}The $Avg^1$-closed sets constitute a topology.\qed
\end{cor}

\begin{ex}The ${\cal{M}}^{acc}$-closed sets do not constitute a topology.
\end{ex}
\begin{proof}Let $H_1=\{\frac{1}{n},1+\frac{1}{n},2+\frac{1}{n}:n\in\mathbb{N}\}\cup\{0,2\},\ H_2=\{3+\frac{1}{n}:n\in\mathbb{N}\}\cup\{3\}$. Then $H_1^{'{\cal{K}}}=\{0,2\},\ H_2^{'{\cal{K}}}=\{3\}$ hence both sets are ${\cal{K}}$-closed for ${\cal{K}}={\cal{M}}^{acc}$. However $H_1\cup H_2$ is not closed since $1\in (H_1\cup H_2)^{'{\cal{K}}}$ but $1\notin H_1\cup H_2$.
\end{proof}

\subsection{Derivative of means}

We can define two derivative type notions for means. The first one measures that how symmetric the set is around $x$ in the sense of the mean.

\begin{df}Let ${\cal{K}}$ be a mean, $H\in Dom\ {\cal{K}}$. Let $x\in\mathbb{R}$ such that $\forall\delta>0\ S(x,\delta)\cap H\in Dom\ {\cal{K}}$ holds. Then set 
\[\overline{d}{\cal{K}}_H(x)=\varlimsup\limits_{\delta\to 0+0}\frac{{\cal{K}}(S(x,\delta)\cap H)-x}{\delta},\ \underline{d}{\cal{K}}_H(x)=\varliminf\limits_{\delta\to 0+0}\frac{{\cal{K}}(S(x,\delta)\cap H)-x}{\delta}\]
and if they are equal then let $d{\cal{K}}_H(x)$ be the common value.
\end{df}

Clearly $-1\leq\overline{d}{\cal{K}}_H(x),\ \underline{d}{\cal{K}}_H(x)\leq 1$. If e.g. $d{\cal{K}}_H(x)=0$ then we can interpret this as $H$ is symmetric in limit around $x$ in the sense of ${\cal{K}}$, while if $d{\cal{K}}_H(x)$ is close to $1$ then it means that $H$ is concentrated mainly on the right hand side of $x$.

We remark that we could have formulated the definition for only the points of $H^{'{\cal{K}}}$ however it would have been too restrictive as e.g. the example of ${\cal{M}}^{iso}$ will show.

\begin{prp}If $H$ is an interval then 
\[dAvg^1_H(x)=
\begin{cases}
\frac{1}{2}&\text{if }x=\inf H\\
-\frac{1}{2}&\text{if }x=\sup H\\
0&\text{if }x\in int(H)
\end{cases}\]
If $H\in Dom(Avg^1)$ arbitrary then $-\frac{1}{2}\leq\overline{d}Avg^1_H(x),\ \underline{d}Avg^1_H(x)\leq \frac{1}{2}$. 
\end{prp}

\begin{prp}Let ${\cal{K}}={\cal{M}}^{acc},\ H\in Dom\ {\cal{K}},\ x\in cl(H)$. Then $d{\cal{K}}_H(x)=0$.
\end{prp}
\begin{proof}Let $n=lev(H)$. Then there is $k\in\mathbb{N}\cup\{0\}$ such that $0\leq k\leq n$ and $x\in H^{(k)}-H^{(k+1)}$. It implies that there is $\delta_0>0$ such that $S(x,\delta_0)\cap H^{(k)}=\{x\}$. Hence if $l>k$ then $S(x,\delta_0)\cap H^{(l)}=\emptyset$ which gives that $lev\big(S(x,\delta_0)\cap H\big)=k$. Clearly if $\delta<\delta_0$ then all previous statements hold as well. Then ${\cal{K}}(S(x,\delta)\cap H)=x$ yields that $d{\cal{K}}_H(x)=0$.
\end{proof}

\begin{ex}Let ${\cal{K}}={\cal{M}}^{iso}$. Then there is $H\in Dom\ {\cal{K}},\ x\in cl(H)$ such that $\overline{d}{\cal{K}}_H(x)=1$.
\end{ex}
\begin{proof}Let $f_1(z)=2^k,\ f_{n+1}(z)=2^{f_n(z)}\ (n\in\mathbb{N},z\in\mathbb{R}^+)$. Let $H_n=\{\frac{1}{n}+\frac{1}{f_n(k)}:k\in\mathbb{N}\}\ (n\in\mathbb{N})$. Finally set $H=\cup_{n=1}^{\infty}H_n$. We show that $\overline{d}{\cal{K}}_H(0)=1$.

Let $g_1(k)=\log_2 k,\ g_{n+1}(k)=\log_2(f_n(k))\ (n,k\in\mathbb{N})$.

Clearly $H_n'=\{\frac{1}{n}\},\ H'=\{\frac{1}{n}:n\in\mathbb{N}\}\cup\{0\}$. First we show that ${\cal{K}}(H)=1$. Let us estimate $\A(H-S(H',\frac{1}{k}))$. Evidently $H-S(H',\frac{1}{k})=\cup_{i=1}^k\{y\in H_i:y-\frac{1}{i}\geq\frac{1}{k}\}$. Let $L_k=\bigcup\limits_{i=2}^k\{y\in H_i:y-\frac{1}{i}\geq\frac{1}{k}\},\ K_k=\{y\in H_1:y-1\geq\frac{1}{k}\}$. Then $|K_k|=[\log_2k],\ |L_k|\leq\sum\limits_{i=2}^kg_i(k)$. Obviously $\A(K_k)\to 1$. If we showed that $\frac{|L_k|}{|K_k|}\to 0$ then we could apply \cite{lamis} 3.2 which would give $\overline{d}{\cal{K}}_H(0)=1$. For given $k\ \exists !m\in\mathbb{N},\exists !v\in\mathbb{R}$ such that $1\leq v<2$ and $k=f_m(v)$. Now one can easily show by induction that if $k\geq f_5(1)$ then $m<\log_2\log_2 k$. But in the sum $\sum\limits_{i=2}^kg_i(k)$ the number of terms is not $k-1$, it is only $m$. Hence $\sum\limits_{i=2}^kg_i(k)<(\log_2\log_2k)^2$ because the maximum term is $\log_2\log_2k$ and there are at most $\log_2\log_2k$ terms. Therefore $\frac{|L_k|}{|K_k|}<\frac{(\log_2\log_2k)^2}{\log_2k-1}\to 0$.

Let us now show that $\overline{d}{\cal{K}}_H(0)=1$. Let $\delta=\frac{1}{n}+\frac{1}{n^2}$. Because $S(x,\delta)\cap H$ has the same structure than $H$ in exactly the same way as before one can show that ${\cal{K}}(S(x,\delta)\cap H)=\frac{1}{n}$. Hence $\lim\limits_{n\to\infty}\frac{{\cal{K}}(S(x,\delta)\cap H)-x}{\delta}\to 1$.
\end{proof}

\begin{rem}Using the same notation one can readily see that $H^{'{\cal{K}}}=\{1\}$ and $d{\cal{K}}_H(1)=0$.\qed
\end{rem}

\medskip
Now we define the second type derivative notion. Throughout this subsection $d(H,K)$ will denote the Hausdorff distance between two compact sets $H,K\subset\mathbb{R}$.

\begin{df}Let ${\cal{K}}$ be a mean defined on some compact sets too. Let $H\in Dom({\cal{K}})$ be compact. Then set
\[\underline{D}{\cal{K}}(H)=\varliminf\limits_{\epsilon\to0+0}\Big\{\frac{{\cal{K}}(K)-{\cal{K}}(H)}{d(H,K)}:K\in Dom({\cal{K}}),\ K\text{ is compact},\ d(H,K)\leq\epsilon\Big\},\]
\[\overline{D}{\cal{K}}(H)=\varlimsup\limits_{\epsilon\to0+0}\Big\{\frac{{\cal{K}}(H)-{\cal{K}}(K)}{d(H,K)}:K\in Dom({\cal{K}}),\ K\text{ is compact},\ d(H,K)\leq\epsilon\Big\},\]
where $\underline{D}{\cal{K}}(H),\ \overline{D}{\cal{K}}(H)$ are the lower and upper derivative of ${\cal{K}}$ at $H$ respectively.
If they are equal then the common value is denoted by $D{\cal{K}}(H)$ and called the derivative of ${\cal{K}}$ at $H$.
\end{df}

We investigate $Avg^1$ in detail.

\begin{lem}\label{lsuphgeqaplp2}Let $H\in Dom(Avg^1)$. Set $m=\sup H,\ a=Avg^1(H),\ l=\lambda(H)$. Then $m\geq a+\frac{l}{2}$. Moreover equality holds if and only if $H=L\cup[a-\frac{l}{2},a+\frac{l}{2}]$ where $L\subset H\cap(-\infty,a-\frac{l}{2})$ and $\lambda(L)=0$.
\end{lem}
\begin{proof}Assume the contrary: $m<a+\frac{l}{2}$. Let $l_1=\lambda(H^{a-}),a_1=Avg^1(H^{a-}),l_2=\lambda(H^{a+}),a_2=Avg^1(H^{a+}).$ Clearly $l=l_1+l_2$. By \ref{cavg1ssii} $l_1,l_2>0$.
Then by \ref{lavgmin}
\[a=\frac{l_1a_1+l_2a_2}{l_1+l_2}\leq\frac{l_1(a-\frac{l_1}{2})+l_2(m-\frac{l_2}{2})}{l}<\frac{l_1(a-\frac{l_1}{2})+l_2(a+\frac{l}{2}-\frac{l_2}{2})}{l}=\]
\[=\frac{l_1(a-\frac{l_1}{2})+l_2(a+\frac{l_1}{2})}{l}=a+\frac{1}{l}\frac{l_1}{2}(l_2-l_1)\]
but by indirect assumption $l_1<l_2$ hence $\frac{1}{l}\frac{l_1}{2}(l_2-l_1)>0$ which is a contradiction.

If $H$ is of the given form then clearly equality holds. If $m=a+\frac{l}{2}$ then suppose that $l_2<\frac{l}{2}<l_1$. Then in almost the same way we get that
\[a=\frac{l_1a_1+l_2a_2}{l_1+l_2}\leq\frac{l_1(a-\frac{l_1}{2})+l_2(m-\frac{l_2}{2})}{l}=\frac{l_1(a-\frac{l_1}{2})+l_2(a+\frac{l}{2}-\frac{l_2}{2})}{l}=\]
\[=\frac{l_1(a-\frac{l_1}{2})+l_2(a+\frac{l_1}{2})}{l}=a+\frac{1}{l}\frac{l_1}{2}(l_2-l_1)<a\]
which is a contradiction hence $l_1=l_2=\frac{l}{2}$. Now suppose that $l_3=\lambda(H\cap(-\infty,a-\frac{l}{2}))>0$. Set $l_4=\lambda(H\cap[a-\frac{l}{2},a])$. Then $l_3+l_4=\frac{l}{2}$. Again the same way
\[a\leq\frac{l_3(a-\frac{l}{2}-\frac{l_3}{2})+l_4(a-\frac{l_4}{2})+\frac{l}{2}(a+\frac{l}{4})}{l}=\]
\[=a+\frac{\frac{l^2}{8}-\frac{l_4^2}{2}-\frac{l}{2}l_3-\frac{l_3^2}{2}}{l}=a+\frac{\frac{l^2}{8}-\frac{1}{2}(l_3+l_4)^2+l_3l_4-\frac{l}{2}l_3}{l}=a+\frac{l_3(l_4-\frac{l}{2})}{l}\]
hence in order not to get a contradiction we must have $l_4=\frac{l}{2}$ i.e. $l_3=0$.
\end{proof}

A similar result can be formulated for $\inf H$.

\begin{lem}\label{lsuphgeqaplp2}Let $H\in Dom(Avg^1)$. Set $m=\inf H,\ a=Avg^1(H),\ l=\lambda(H)$. Then $m\leq a-\frac{l}{2}$. Moreover equality holds if and only if $H=L\cup[a-\frac{l}{2},a+\frac{l}{2}]$ where $L\subset H\cap(a+\frac{l}{2},+\infty)$ and $\lambda(L)=0$.\qed
\end{lem}

\begin{thm}Let ${\cal{K}}=Avg^1,\ H\in Dom({\cal{K}})$ be compact. Then $\overline{D}{\cal{K}}(H)\geq\frac{1}{2}$ and equality holds if and only if $H=L\cup[a-\frac{l}{2},a+\frac{l}{2}]$ where $L\subset H\cap(-\infty,a-\frac{l}{2})$ and $\lambda(L)=0$ where $a=Avg^1(H),\ l=\lambda(H)$.
\end{thm}
\begin{proof}Set $m=\sup H$. Take any point $b$ such that $l_1=\lambda(H^{b-})>0$ and $l_2=\lambda(H^{b+})>0$. Set $a_1=Avg^1(H^{b-}),a_2=Avg^1(H^{b+})$. Let $\epsilon>0$. Then
\[Avg^1(H\cup[m,m+\epsilon])=\frac{l_1a_1+l_2a_2+\epsilon(m+\frac{\epsilon}{2})}{l_1+l_2+\epsilon},\ Avg^1(H)=\frac{l_1a_1+l_2a_2}{l_1+l_2}.\]
\[\frac{Avg^1(H\cup[m,m+\epsilon])-Avg^1(H)}{\epsilon}=\frac{ml_1+ml_2-l_1a_1-l_2a_2+l_1\frac{\epsilon}{2}+l_2\frac{\epsilon}{2}}{(l_1+l_2)(l_1+l_2+\epsilon)}\]
and if $\epsilon\to0+0$ then we get \[\frac{ml_1+ml_2-l_1a_1-l_2a_2}{l^2}=\frac{m}{l}-\frac{Avg^1(H)}{l}=\frac{1}{l}(m-Avg^1(H))\geq\frac{1}{2}\]
by \ref{lsuphgeqaplp2} and equality holds if and only if $H=L\cup[a-\frac{l}{2},a+\frac{l}{2}]$ where $L\subset H\cap(-\infty,a-\frac{l}{2})$ and $\lambda(L)=0$.
\end{proof}

Similarly one can prove

\begin{thm}Let ${\cal{K}}=Avg^1,\ H\in Dom({\cal{K}})$ be compact. Then $\underline{D}{\cal{K}}(H)\leq-\frac{1}{2}$ and equality holds if and only if $H=L\cup[a-\frac{l}{2},a+\frac{l}{2}]$ where $L\subset H\cap(a+\frac{l}{2},+\infty)$ and $\lambda(L)=0$ where $a=Avg^1(H),\ l=\lambda(H)$.\qed
\end{thm}

One might wonder if there is a finite upper limit for the upper derivatives. The answer is negative as the next example shows. 

\begin{ex}Let ${\cal{K}}=Avg^1$. Then $\sup\{\overline{D}{\cal{K}}(H):H\in Dom({\cal{K}})\}=+\infty$.
\end{ex}
\begin{proof}Let $a,b\in\mathbb{R},\ a>0,a<b$. Let $H_{a,b}=H=[0,a]\cup[b,b+a]$. Let $\epsilon>0$ and $H_{\epsilon}=H\cup[b+a,b+a+\epsilon]$. Then one can easily check that $Avg^1(H)=\frac{a+b}{2}$. In the usual way we get that
\[Avg^1(H_{\epsilon})=\frac{\frac{a}{2}a+\frac{2b+a+\epsilon}{2}(a+\epsilon)}{2a+\epsilon}.\]
\[Avg^1(H_{\epsilon})-Avg^1(H)=\frac{b\epsilon+a\epsilon+\epsilon^2}{2(2a+\epsilon)},\]
hence
\[\lim\limits_{\epsilon\to0+0}\frac{Avg^1(H_{\epsilon})-Avg^1(H)}{\epsilon}=\frac{a+b}{2}\frac{1}{2a},\]
which means that $\overline{D}{\cal{K}}(H_{a,b})\geq\frac{a+b}{2}\frac{1}{2a}$.
Now if we take $a$ fixed and $b$ tends to infinity then $\overline{D}{\cal{K}}(H_{a,b})$ must tend to infinity as well.
\end{proof}

\subsection{On continuity}


\begin{df}${\cal{K}}$ is called \textbf{u-Cantor-continuous} if $H,H_i\in Dom({\cal{K}})\ (i\in\mathbb{N}), H_i\cap H_j=\emptyset\ (i\ne j), H=\bigcup\limits_{i=1}^{\infty}H_i$. Then
$\lim\limits_{n\to\infty}{\cal{K}}(\bigcup\limits_{i=1}^{n}H_i)={\cal{K}}(H)$.
\end{df}

\cite{lambm} Lemma 2.19 gives that ${\cal{M}}^{\mu}$ is u-Cantor-continuous.

\begin{df}${\cal{K}}$ is called \textbf{u-bounded} if $H,H_1,H_2\in Dom\ {\cal{K}},\ H_1\cap H_2=\emptyset$ bounded sets then 
$$|{\cal{K}}(H)-{\cal{K}}(H\cup H_1\cup H_2)|\leq|{\cal{K}}(H)-{\cal{K}}(H\cup H_1)|+|{\cal{K}}(H)-{\cal{K}}(H\cup H_2)|.$$
\end{df}

\begin{rem}We can assume that $H\cap H_i=\emptyset\ (i=1,2)$ since if we consider $\hat{H_i}=H_i-H$ instead of $H_i\ (i=1,2)$ then we end up with the same inequality.
\end{rem}

\begin{prp}\label{pubns}If ${\cal{K}}$ is u-bounded, $H,H_i\in Dom\ {\cal{K}}\ (1\leq i\leq n,\ n\in\mathbb{N}),\ H_i\cap H_j=\emptyset\ (i\ne j)$ bounded sets then 
$$|{\cal{K}}(H)-{\cal{K}}(H\cup\bigcup_{i=1}^{n} H_i)|\leq\sum\limits_{i=1}^{n}|{\cal{K}}(H)-{\cal{K}}(H\cup H_i)|.$$
\end{prp}
\begin{proof}We prove it by induction. It is true for 2 sets by definition. Assume it holds for $n$ sets. then clearly 
$$|{\cal{K}}(H)-{\cal{K}}(H\cup\bigcup_{i=1}^{n} H_i\cup H_{n+1})|\leq|{\cal{K}}(H)-{\cal{K}}(H\cup \bigcup_{i=1}^{n} H_i)|+|{\cal{K}}(H)-{\cal{K}}(H\cup H_{n+1})|$$
$$\leq\sum\limits_{i=1}^{n}|{\cal{K}}(H)-{\cal{K}}(H\cup H_i)|+|{\cal{K}}(H)-{\cal{K}}(H\cup H_{n+1})|.$$
\end{proof}

\begin{prp}If ${\cal{K}}$ is u-bounded and u-Cantor-continuous, $H,H_i,\bigcup\limits_{i=1}^{\infty} H_i\in Dom\ {\cal{K}}\ (i\in\mathbb{N}),\ H_i\cap H_j=\emptyset\ (i\ne j)$ bounded sets then 
$$|{\cal{K}}(H)-{\cal{K}}(H\cup\bigcup_{i=1}^{\infty} H_i)|\leq\sum\limits_{i=1}^{\infty}|{\cal{K}}(H)-{\cal{K}}(H\cup H_i)|.$$
\end{prp}
\begin{proof}We can assume that $H\cap H_i=\emptyset\ (i\in\mathbb{N})$. Using \ref{pubns} we have 
\begin{equation}\label{eq1}
\left\lvert{\cal{K}}(H)-{\cal{K}}\Big(H\cup\bigcup_{i=1}^{n} H_i\Big)\right\lvert\leq\sum\limits_{i=1}^{n}|{\cal{K}}(H)-{\cal{K}}(H\cup H_i)|\leq\sum\limits_{i=1}^{\infty}|{\cal{K}}(H)-{\cal{K}}(H\cup H_i)|.
\end{equation}
 U-Cantor-continuity gives that $\lim\limits_{n\to\infty}{\cal{K}}(H\cup\bigcup\limits_{i=1}^{n} H_i)={\cal{K}}(H\cup\bigcup\limits_{i=1}^{\infty} H_i)$ because modify the first set to $H\cup H_1$ in the definition of u-Cantor-continuity.

When $n$ tends to infinity in (\ref{eq1}) we get the statement.
\end{proof}

\begin{ex}If we want u-boundedness to be valid for basic means then we cannot omit the disjointness condition in the definition of u-boundedness.

We show it for the arithmetic mean first. Let ${\cal{K}}$ denote the arithmetic mean.
Let $H_0=\{0\},H_1=\{1,-1\},H_2=\{-1,2\}$. Then $|{\cal{K}}(H)-{\cal{K}}(H\cup H_1\cup H_2)|\leq|{\cal{K}}(H)-{\cal{K}}(H\cup H_1)|+|{\cal{K}}(H)-{\cal{K}}(H\cup H_2)|$ does not hold since $|{\cal{K}}(H)-{\cal{K}}(H\cup H_1\cup H_2)|=0.5,\ |{\cal{K}}(H)-{\cal{K}}(H\cup H_1)|=0,|{\cal{K}}(H)-{\cal{K}}(H\cup H_2)|=\frac{1}{3}$.

To get an example for $Avg^1$ simply put small $\epsilon$ neighbourhoods around the points $-1,0,1,2$ and use those intervals instead of the points to create the same sets. When $\epsilon$ tends to 0 then $Avg^1$ tends to the arithmetic mean (see \cite{lamisii} Lemma 6) hence the inequality for $Avg^1$ cannot hold either. 
\end{ex}

\begin{ex}${\cal{M}}^{\mu}$ is u-bounded.
\end{ex}
\begin{proof}Let $H,H_1,H_2\in Dom\ {\cal{M}}^{\mu},\ H_1\cap H_2=\emptyset$ bounded. Set ${\cal{K}}={\cal{M}}^{\mu}$. We can assume that $H\cap H_i=\emptyset\ (i=1,2)$.
$$|{\cal{K}}(H)-{\cal{K}}(H\cup H_1\cup H_2)|=\left\lvert\frac{\mu(H){\cal{K}}(H)+\mu(H_1){\cal{K}}(H_1)+\mu(H_2){\cal{K}}(H_2)}{\mu(H)+\mu(H_1)+\mu(H_2)}-{\cal{K}}(H)\right\lvert\leq$$
$$\frac{\mu(H_1)|{\cal{K}}(H_1)-{\cal{K}}(H)|+\mu(H_2)|{\cal{K}}(H_2)-{\cal{K}}(H)|}{\mu(H)+\mu(H_1)+\mu(H_2)}\leq$$
$$\frac{\mu(H_1)|{\cal{K}}(H_1)-{\cal{K}}(H)|+\mu(H_2)|{\cal{K}}(H_2)-{\cal{K}}(H)|}{\mu(H_1)+\mu(H_2)}.$$
But similarly we get that $$|{\cal{K}}(H)-{\cal{K}}(H\cup H_i)|=\frac{\mu(H_i)|{\cal{K}}(H_i)-{\cal{K}}(H)|}{\mu(H_1)+\mu(H_2)}\ \ (i=1,2)$$ which gives the statement.
\end{proof}

\section{Constructing new means from old ones}

\begin{df}Let $f:\mathbb{R}\to\mathbb{R}$ be a strictly monotone continuous function. Let ${\cal{K}}$ be a mean. If $H\in Dom({\cal{K}})$ set ${\cal{K}}^f(H)=f^{-1}({\cal{K}}(f(H)))$.
\end{df}

\begin{prp}Let $f:\mathbb{R}\to\mathbb{R}$ be a strictly monotone continuous function. If ${\cal{K}}$ is internal, strictly-internal, monotone, mean-monotone, union-monotone, slice-continuous, point-continuous, Cantor-continuous, finite-independent then so is ${\cal{K}}^f$.
\end{prp}
\begin{proof}First let us note some basic facts regarding $f$. Obviously $f$ preserves the order, $A\subset B\Rightarrow f(A)\subset f(B)$, $f(\cup A_i)=\cup f(A_i),f(\cap A_i)=\cap f(A_i)$ (for any system of $A_i$). Hence if $A,B$ are disjoint then so are $f(A),f(B)$. Clearly $f(\inf A)=\inf f(A),f(\sup A)=\sup f(A),f(\varliminf A)=\varliminf f(A),f(\varlimsup A)=\varlimsup f(A)$. Evidently $|f(A)|=|A|$. 

Let us assume that $f$ is increasing (the other case is similar).

If ${\cal{K}}$ is internal then so is ${\cal{K}}^f$: $\inf f(H)\leq{\cal{K}}(f(H))\leq\sup f(H)$ and applying $f^{-1}$ we get that $\inf H=f^{-1}(\inf f(H))\leq f^{-1}({\cal{K}}(f(H)))\leq f^{-1}(\sup f(H))=\sup H$.

If ${\cal{K}}$ is strict-internal then so is ${\cal{K}}^f$: it can be proved as internality, just replace $\inf,\sup$ with $\varliminf,\varlimsup$.

If ${\cal{K}}$ is monotone then so is ${\cal{K}}^f$: Let $\sup H_1\leq\inf H_2$. Then $\sup f(H_1)\leq\inf f(H_2)$ which gives that ${\cal{K}}(f(H_1))\leq{\cal{K}}(f(H_1)\cup f(H_2))\leq{\cal{K}}(f(H_2))$. Using that $f(H_1)\cup f(H_2)=f(H_1\cup H_2)$ and applying $f^{-1}$ we get the statement.

If ${\cal{K}}$ is mean-monotone then so is ${\cal{K}}^f$: Let $\sup K_1\leq f^{-1}({\cal{K}}(f(H)))\leq\inf K_2$. Then $f(\inf K_1)=\inf f(K_1)\leq{\cal{K}}(f(H))\leq\sup f(K_2)=f(\sup K_2)$ which gives that ${\cal{K}}(f(H\cup K_1))={\cal{K}}(f(H)\cup f(K_1))\leq{\cal{K}}(f(H))\leq{\cal{K}}(f(H)\cup f(K_2))={\cal{K}}(f(H\cup K_2))$. Applying $f^{-1}$ gives the assertion.

If ${\cal{K}}$ is union-monotone then so is ${\cal{K}}^f$: Assume that $B\cap C=\emptyset$ and ${\cal{K}}^f(A)\leq{\cal{K}}^f(A\cup B),{\cal{K}}^f(A)\leq{\cal{K}}^f(A\cup C)$. Then $f(B)\cap f(C)=\emptyset$ and ${\cal{K}}(f(A))\leq{\cal{K}}(f(A)\cup f(B)),{\cal{K}}(f(A))\leq{\cal{K}}(f(A)\cup f(C))$ which implies that ${\cal{K}}(f(A))\leq{\cal{K}}(f(A)\cup f(B)\cup f(C))$. From that we get that ${\cal{K}}^f(A)\leq{\cal{K}}^f(A\cup B\cup C)$. The other inequality is similar.

If ${\cal{K}}$ is slice-continuous then so is ${\cal{K}}^f$: First observe that $f(H^{x-})=f(H)^{f(x)-}$. We have to show that ${\cal{K}}^f(H^{x-})=f^{-1}({\cal{K}}(f(H)^{f(x)-}))$ is continuous. This follows from that $x\mapsto {\cal{K}}(f(H)^{f(x)-})$ is continuous. The "+" case is similar.

If ${\cal{K}}$ is point-continuous then so is ${\cal{K}}^f$: Let $H\in Dom({\cal{K}})$ and $x\in\mathbb{R}$. We know that $\lim\limits_{\epsilon\to 0+0}{\cal{K}}(f(H-S(x,\epsilon)))=\lim\limits_{\epsilon\to 0+0}{\cal{K}}(f(H)-S(x,\epsilon))={\cal{K}}(f(H)).$ Now by continuity of $f^{-1}$ we get that 
\[f^{-1}(\lim\limits_{\epsilon\to 0+0}{\cal{K}}(f(H-S(x,\epsilon))))=\lim\limits_{\epsilon\to 0+0}f^{-1}({\cal{K}}(f(H-S(x,\epsilon))))\]
which gives that $\lim\limits_{\epsilon\to 0+0}{\cal{K}}^f(H-S(x,\epsilon))={\cal{K}}^f(H)$.

If ${\cal{K}}$ is Cantor-continuous then so is ${\cal{K}}^f$: Let $H_{i+1}\subset H_i$. Then $f(H_{i+1})\subset f(H_i)$, ${\cal{K}}(f(H_i))\to{\cal{K}}(\cap f(H_i))={\cal{K}}(f(\cap H_i))$. If we apply $f^{-1}$ for both sides then we get the statement.

Finite-independence: $|f(K)|=|K|$ implies that ${\cal{K}}(f(H-K))={\cal{K}}(f(H)-f(K))={\cal{K}}(f(H))$ if $K$ is finite. Similarly for union.
\end{proof}

\begin{ex}Let ${\cal{K}}$ be disjoint monotone. Then $f:\mathbb{R}\to\mathbb{R}$ being strictly monotone and continuous does not imply that ${\cal{K}}^f$ is disjoint monotone.
\end{ex}
\begin{proof}Let ${\cal{K}}=Avg^1,\ f(x)=x^2, H_1=[0,1]\cup[2,3],H_2=(1,2)$. Then $H_1\cap H_2=\emptyset,{\cal{K}}(H_1)\leq{\cal{K}}(H_2)$ holds. But ${\cal{K}}(f(H_1\cup H_2))=4.5\not\leq {\cal{K}}(f(H_2))=2.5$.
\end{proof}

\begin{df}Let $f:\mathbb{R}\to\mathbb{R}$ be a strictly monotone continuous function. If $H$ is of positive Lebesgue measure then set $Avg_f(H)=f^{-1}\Big(\frac{\int\limits_H fd\lambda}{\lambda(H)}\Big)$.
\end{df}

\begin{prp}$Avg_f$ is strictly-internal, disjoint-monotone, Cantor-continuous, finite-independent.
\end{prp}
\begin{proof}Let us assume that $f$ is increasing (the other case is similar).

Strict-internal: Clearly $\lambda(H)f(\varliminf H)\leq\int\limits_H fd\lambda\leq\lambda(H)f(\varlimsup H)$ which gives the statement.

Disjoint-monotone: Let $H_1\cap H_2=\emptyset,\ \frac{\int\limits_{H_1} fd\lambda}{\lambda(H_1)}\leq\frac{\int\limits_{H_2} fd\lambda}{\lambda(H_2)}$. We need $\frac{\int\limits_{H_1} fd\lambda}{\lambda(H_1)}\leq\frac{\int\limits_{H_1\cup H_2} fd\lambda}{\lambda(H_1\cup H_2)}=\frac{\int\limits_{H_1} fd\lambda+\int\limits_{H_2} fd\lambda}{\lambda(H_1)+\lambda(H_2)}$. Let us multiply both sides with $\lambda(H_1)(\lambda(H_1)+\lambda(H_2))$. Then we end up with $\lambda(H_2){\int\limits_{H_1} fd\lambda}\leq\lambda(H_1){\int\limits_{H_2} fd\lambda}$ which is equivalent to the assumption.

Cantor-continuous: Let $H=\cap_1^{\infty}H_n$. Then \[\left\lvert\frac{\int\limits_{H_n} fd\lambda}{\lambda(H_n)}-\frac{\int\limits_H fd\lambda}{\lambda(H)}\right\lvert
=\left\lvert\frac{\lambda(H)(\int\limits_{H_n} fd\lambda-\int\limits_{H} fd\lambda)+(\lambda(H)-\lambda(H_n))\int\limits_{H} fd\lambda}{\lambda(H_n)\lambda(H)}\right\lvert\to 0\]
because $\lambda(H_n)\to\lambda(H)$.

Finite-independent: The integral has this property.
\end{proof}

$Avg_f$ being disjoint-monotone shows that $Avg^f\ne Avg_f$ however we also give an example to prove that.

\begin{ex}$Avg^f\ne Avg_f$. For $f(x)=x^2$ we get that $Avg^f([a,b])=\sqrt{\frac{a^2+b^2}{2}}$ while $Avg_f([a,b])=\sqrt{\frac{a^2+ab+b^2}{3}}$.
\end{ex}

\section{Limit of means}

We have already defined and investigated means that were constructed via limit of means. Here we present some such examples. 
\begin{enumerate}
\item If $H\subset\mathbb{R}$ then $LAvg(H)=\lim\limits_{n\to\infty}Avg(S(H,\frac{1}{n}))$ (see \cite{lamis} Section 5).

\item  If $H\subset\mathbb{R}$ then ${\cal{M}}^{eds}(H)=\lim\limits_{n\to\infty}\A(\{a+\frac{i}{n}(b-a):H_{n,i}\ne\emptyset\})$ where $a=\inf H,b=\sup H,\ H_{n,i}=H\cap[a+\frac{i}{n}(b-a),a+\frac{i+1}{n}(b-a))$ (see \cite{lamis} Section 6).

\item If $cl(H-H')=H$ then let ${\cal{M}}^{iso}(H)=\lim\limits_{n\to\infty}\A(H-S(H',\frac{1}{n}))$ (see \cite{lamis} Section 3.1).

\item If $({\cal{I}}_k)$ be an ascending sequence of ideals then ${\cal{M}}^{({\cal{I}}_k)}=\lim\limits_{n\to\infty}{\cal{M}}^{({\cal{I}}_1,\dots,{\cal{I}}_n)}$  (see \cite{lamis} Section 3.3).
\end{enumerate}

If we examine these examples carefully then we can see the underlying more basic means:
\begin{enumerate}
\item $Avg(S(H,\frac{1}{n}))$ is a mean for $n\in\mathbb{N}$.

\item $\A(\{a+\frac{i}{n}(b-a):H_{n,i}\ne\emptyset\})$  is a mean for $n\in\mathbb{N}$.

\item $\A(H-S(H',\frac{1}{n}))$ is a mean for $n\in\mathbb{N}$ where $H-S(H',\frac{1}{n})\ne\emptyset$.

\item ${\cal{M}}^{({\cal{I}}_1,\dots,{\cal{I}}_n)}$ is a mean for $n\in\mathbb{N}$.
\end{enumerate}

\subsection{Some properties of underlying means}
\begin{prp}Let $\delta>0$. Then ${\cal{K}}(H)=Avg(S(H,\delta))$ is a mean. ${\cal{K}}$ is monotone, closed, slice-continuous.
\end{prp}
\begin{proof}Clearly $S(H,\delta)=S(\inf H,\delta)\cup^* (S(H,\delta)-S(\inf H,\delta))$. Using that $Avg$ is monotone it gives that $\inf H=Avg(S(\inf H,\delta))\leq Avg(S(H,\delta))$. The other inequality is similar.

Monotone: Let $\sup H_1\leq\inf H_2$. Then $S(H_1\cup H_2,\delta)=S(H_1,\delta)\cup^*(S(H_2,\delta)-S(H_1,\delta))$ and $\sup S(H_1,\delta)\leq\inf S(H_2,\delta)-S(H_1,\delta)$. Using the monotonicity of $Avg$ gives the statement.

Closed: It is the consequence of $S(H,\delta)=S(cl(H),\delta)$.

Slice-continuous: Obviously \[|\lambda(S(H^{x-},\delta))-\lambda(S(H^{y-},\delta))|\leq\lambda(S(H\cap[x,y],\delta)).\] If $x_n\to x$ then $\lambda(S(H^{x_n-}))\to\lambda(S(H^{x-}))$ and $\lambda(S(H\cap[x_n,x]))\to 0$ which gives that $Avg(S(H^{x_n-},\delta))\to Avg(S(H^{x-},\delta))$. The "+" case is similar.
\end{proof}

\begin{ex}Let $\delta>0$. Then ${\cal{K}}(H)=Avg(S(H,\delta))$ is not strict-internal, not finite-independent, not Cantor-continuous.
\end{ex}
\begin{proof}Let $H=\{0\}\cup[2\delta,3\delta]$. This shows that ${\cal{K}}$ is not strict-internal, not finite-independent.

Let $H=\{1+2\delta\}, (q_n)$ is a sequence of all rational numbers between 0 and 1, $H_n=H \cup\{q_i:i\geq n\}$. Clearly $H=\cap H_n$ and ${\cal{K}}(H_m)={\cal{K}}(H_n)$. But ${\cal{K}}(H)\ne{\cal{K}}(H_n)$ showing that ${\cal{K}}$ not Cantor-continuous.
\end{proof}

However a  stronger version of Cantor-continuity holds.

\begin{prp}Let $\delta>0,\ {\cal{K}}(L)=Avg(S(L,\delta))$. If $H_n$ is compact $(n\in\mathbb{N})$ and $H_{n+1}\subset H_n,\ H=\cap_1^{\infty}H_n$ then ${\cal{K}}(H_n)\to {\cal{K}}(H)$.
\end{prp}
\begin{proof}We show that $S(H,\delta)\subset\cap_1^{\infty}S(H_n,\delta)\subset cl(S(H,\delta))$. The first inclusion is trivial. We show the second inclusion. Let $x\in \cap_1^{\infty}S(H_n,\delta)$. Then there is $x_n\in H_n$ such that $x\in S(x_n,\delta)$. From $(x_n)$ one can choose a convergent subsequence $(x_{n_k})$. Let $x_{n_k}\to x'$. Then $x'\in H_n$ for all $n$ hence $x'\in H$ and evidently $|x-x'|\leq\delta$. Therefore $x\in cl(S(x',\delta))\subset cl(S(H,\delta))$.

We show that $\lambda(cl(S(H,\delta))-S(H,\delta))=0$. By $H$ being compact there are $x_1,\dots,x_n\in H$ such that $\cup_{i=1}^nS(x_i,\delta)=S(H,\delta)$ that gives that $cl(S(H,\delta))=\cup_{i=1}^ncl(S(x_i,\delta))$ hence $cl(S(H,\delta))-S(H,\delta)$ is finite.

Obviously $S(H_{n+1},\delta)\subset S(H_n,\delta)$ which yields that $Avg(S(H_n,\delta))\to Avg(\cap_1^{\infty}S(H_n,\delta))$ by Cantor-continuity of $Avg$. But $Avg(\cap_1^{\infty}S(H_n,\delta))=Avg(S(H,\delta))$ by \cite{lambm} Lemma 2.5.
\end{proof}

\begin{df}${\cal{K}}$ is called Hausdorff-continuous if $H_n,H$ are compact sets $(n\in\mathbb{N})$, $H_n\to H$ in the Hausdorff metric then ${\cal{K}}(H_n)\to{\cal{K}}(H)$. 
\end{df}

\begin{prp}\label{plavgnhc}$Avg,\ LAvg$ are not Hausdorff-continuous.
\end{prp}
\begin{proof}Let $H=[0,2],\ H_n=[1,2]\cup\{\frac{k}{n}:0\leq k\leq n,k\in\mathbb{N}\}$. Clearly $H_n,H$ are compact, $H_n\to H$ in the Hausdorff metric but $Avg(H)=1,\ Avg(H_n)=1.5$ for all $n\in\mathbb{N}$. By \cite{lambm} Theorem 5.8 $Avg=LAvg$ for closed sets. 
\end{proof}

\begin{lem}\label{lssp}Let $H\subset\mathbb{R},\ \epsilon, \delta>0$. Then $S(S(H,\epsilon),\delta)=S(H,\epsilon+\delta)$.
\end{lem}
\begin{proof}Let $x\in S(S(H,\epsilon),\delta)$. Then there is $y\in\mathbb{R},h\in H$ such that $x\in S(y,\delta), y\in S(h,\epsilon)$ which gives that $|h-x|<\epsilon+\delta$.

Let $x\in S(H,\epsilon+\delta)$. Then there is $h\in H$ such that $|h-x|<\epsilon+\delta$ which yields that there is $y\in S(x,\delta)\cap S(h,\epsilon)\ne\emptyset$.
\end{proof}

\begin{lem}\label{llshde}Let $\delta>0,\ H\subset\mathbb{R}$ be compact. Then \[\lim\limits_{\epsilon\to 0+0}\lambda(S(H,\delta+\epsilon)-S(H,\delta-\epsilon))=0.\]
\end{lem}
\begin{proof}Let $0<\epsilon<\delta$. 

By compactness there are points $x_1,\dots,x_n$ such that $S(H,\delta)=\cup_{i=1}^nS(x_i,\delta)$. Obviously $S(H,\delta+\epsilon)=\cup_{i=1}^nS(x_i,\delta+\epsilon)$ hence $\lambda(S(H,\delta+\epsilon)-S(H,\delta))\leq 2\epsilon n$.

Clearly $\cup_{i=1}^nS(x_i,\delta-\epsilon)\subset S(H,\delta-\epsilon)\subset S(H,\delta)$ which gives that $\lambda(S(H,\delta)-S(H,\delta-\epsilon))\leq\lambda(S(H,\delta)-\cup_{i=1}^nS(x_i,\delta-\epsilon))\leq 2\epsilon n$.
\end{proof}

\begin{thm}\label{tavgsdhc}Let $\delta>0,\ {\cal{K}}(L)=Avg(S(L,\delta))$. Then ${\cal{K}}$ is Hausdorff-continuous.
\end{thm}
\begin{proof}Let $0<\epsilon<\delta$. Then there is $N\in\mathbb{N}$ such that $n>N$ implies that $H_n\subset S(H,\epsilon),H\subset S(H_n,\epsilon)$. By \ref{lssp} we get that $S(H,\delta-\epsilon)\subset S(H_n,\delta)\subset S(H,\delta+\epsilon)$. Obviously $S(H,\delta-\epsilon)\subset S(H,\delta)\subset S(H,\delta+\epsilon)$ holds as well. By \ref{llshde} $\lim\limits_{\epsilon\to 0+0}\lambda((S(H,\delta)-S(H_n,\delta))\cup(S(H_n,\delta)-S(H,\delta)))=0$. By \cite{lambm} Lemma 2.15 we get the statement.
\end{proof}

\begin{prp}Let $n\in\mathbb{N},\ H\subset\mathbb{R}$, $a=\inf H,b=\sup H,\ H_{n,i}=H\cap\big[a+\frac{i}{n}(b-a),a+\frac{i+1}{n}(b-a)\big)\ (i\in\mathbb{Z})$. Then ${\cal{K}}(H)=\A(\{a+\frac{i}{n}(b-a):H_{n,i}\ne\emptyset\})$ is a mean. 
\end{prp}
\begin{proof}Clearly $H_{n,i}\ne\emptyset$ implies that $\inf H\leq a+\frac{i}{n}(b-a)\leq\sup H$ hence ${\cal{K}}$ is a mean.
\end{proof}

\begin{ex}Let $n\in\mathbb{N},\ H\subset\mathbb{R}$, $a=\inf H,b=\sup H,\ H_{n,i}=H\cap\big[a+\frac{i}{n}(b-a),a+\frac{i+1}{n}(b-a)\big)\ (i\in\mathbb{Z})$. Then ${\cal{K}}(H)=\A(\{a+\frac{i}{n}(b-a):H_{n,i}\ne\emptyset\})$ is not closed, not slice-continuous, not finite-independent, not strict-internal.
\end{ex}
\begin{proof}Let $H=\{0,3\}\cup(1,2)$ and $n=3$. Then ${\cal{K}}(H)=\A(\{0,1,3\})$ while ${\cal{K}}(cl(H))=\A(\{0,1,2,3\})$ showing that ${\cal{K}}$ is not closed.

Let $L=\{0,1,2,3\},\ n=3$. Let $x_i\to 3,x_i<3$. Then $\lim\limits_{i\to\infty}{\cal{K}}(L^{x_i-})=\A(\{0,1,2\})$ and ${\cal{K}}(L)=\A(\{0,1,2,3\})$ therefore ${\cal{K}}$ is not slice-continuous.

Obviously ${\cal{K}}(L)\ne{\cal{K}}(L-\{0\})$ hence ${\cal{K}}$ is not finite-independent.

Taking any $n$, ${\cal{K}}(\{\frac{1}{i}:i\in\mathbb{N}\})\ne 0$ proving that ${\cal{K}}$ is not strict internal.
\end{proof}

\begin{prp}Let $n\in\mathbb{N}$. Then $Dom({\cal{K}})=\{H\subset\mathbb{R}:cl(H-H')=H,\ H-S(H',\frac{1}{n})\ne\emptyset\},\ {\cal{K}}(H)=\A(H-S(H',\frac{1}{n}))$ is a mean.
\end{prp}
\begin{proof}Obviously $h\in H-S(H',\frac{1}{n})$ implies that $\inf H\leq h\leq\sup H$.
\end{proof}

\begin{ex}Let $n\in\mathbb{N}$. Then $Dom({\cal{K}})=\{H\subset\mathbb{R}:cl(H-H')=H,\ H-S(H',\frac{1}{n})\ne\emptyset\},\ {\cal{K}}(H)=\A(H-S(H',\frac{1}{n}))$ is not monotone.
\end{ex}
\begin{proof}Let $H_1=\{0,2(n+3),2(n+3)-\frac{1}{2n},\frac{1}{k}:k\geq n\},\ H_2=\{2(n+3),2(n+3)+\frac{1}{k}:k\geq n\}$. Clearly $cl(H_1-H_1')=H_1,cl(H_2-H_2')=H_2,\sup H_1\leq\inf H_2$. Then $\A(H_1-S(H_1',\frac{1}{n}))=\A\{\frac{1}{n},2(n+3)-\frac{1}{2n},2(n+3)\}=\frac{4(n+3)}{3}+\frac{1}{6n}$ and $\A(H_1\cup H_2-S(H_1'\cup H_2',\frac{1}{n}))=\A\{\frac{1}{n},2(n+3)+\frac{1}{n}\}=n+3+\frac{1}{n}$. Evidently $\frac{4(n+3)}{3}+\frac{1}{6n}>n+3+\frac{1}{n}$ showing that ${\cal{K}}$ is not monotone.
\end{proof}

\subsection{General limits}
\begin{df}We say that a sequence of generalized means $({\cal{K}}_i)$ converge pointwise to a mean ${\cal{K}}$ if for all $H\in \cap_1^{\infty}Dom({\cal{K}}_i)\cap Dom({\cal{K}})$ ${\cal{K}}_i(H)\to{\cal{K}}(H)$. In this case we use the usual notation ${\cal{K}}_i\to{\cal{K}}$. 
\end{df}

\begin{df}We say that a sequence of generalized means $({\cal{K}}_i)$ converge uniformly to a mean ${\cal{K}}$ if $\forall\epsilon>0\ \exists N\in\mathbb{N}$ such that $n>N$ implies that for all $H\in \cap_1^{\infty}Dom({\cal{K}}_i)\cap Dom({\cal{K}})$ $|{\cal{K}}_n(H)-{\cal{K}}(H)|<\epsilon$. 
\end{df}

\begin{prp}Let ${\cal{K}}_i\to{\cal{K}},\ H\subset\mathbb{R}$, and $H\in \cap_1^{\infty}Dom({\cal{K}}_i)\cap Dom({\cal{K}})$. Furthermore $\forall h\in H\ H-\{h\}\in Dom({\cal{K}}_i)\cap Dom({\cal{K}})$ and $|{\cal{K}}_i(H)-{\cal{K}}_i(H-\{h\})|\to 0$. Then ${\cal{K}}$ is finite-independent.
\end{prp}
\begin{proof}Clearly ${\cal{K}}_i(H)\to{\cal{K}}(H),\ {\cal{K}}_i(H-\{h\})\to{\cal{K}}(H-\{h\})$ hence ${\cal{K}}(H-\{h\})={\cal{K}}(H)$. From that we get the statement by induction.
\end{proof}

\begin{thm}If ${\cal{K}}_i\to{\cal{K}}$ and for all $i\in\mathbb{N}\ {\cal{K}}_i$ is internal, strong-internal, finite-independent, monotone, disjoint-monotone, closed, accumulated, convex, translation-invariant, reflection-invariant, homogeneous respectively then so is ${\cal{K}}$.
\end{thm}
\begin{proof}Internal: If $\forall i\ \inf H\leq {\cal{K}}_i(H)\leq\sup H$ then $\inf H\leq {\cal{K}}(H)\leq\sup H$.

Strong-internal: Replace $\inf,\sup$ with $\varliminf,\varlimsup$ in the proof of internal.

Finite-independent: If $\forall i\ {\cal{K}}_i(H)={\cal{K}}_i(H-V)$ then ${\cal{K}}(H)={\cal{K}}(H-V)$ where $V$ is finite. The same for union.

Monotone: If $\sup H_1\leq\inf H_2$ then $\forall i\ {\cal{K}}_i(H_1)\leq{\cal{K}}_i(H_1\cup H_2)\leq{\cal{K}}_i(H_2)$ which gives the same for ${\cal{K}}$.

Disjoint-monotone: Let $H_1\cap H_2=\emptyset$. If ${\cal{K}}(H_1)<{\cal{K}}(H_2)$ then there is $N\in\mathbb{N}$ such that $n>N$ implies that ${\cal{K}}_n(H_1)<{\cal{K}}_n(H_2)$. From that we get the statement similarly as for monotonicity. If ${\cal{K}}(H_1)={\cal{K}}(H_2)$ then there is infinitely many $n$ such that either ${\cal{K}}_n(H_1)\geq{\cal{K}}_n(H_2)$ or ${\cal{K}}_n(H_1)\leq{\cal{K}}_n(H_2)$ holds. In the first case ${\cal{K}}_n(H_2)\leq{\cal{K}}_n(H_1\cup H_2)\leq{\cal{K}}_n(H_1)$ holds from which the assertion follows. The second case is similar.

Closed, accumulated:  If $\forall i\ {\cal{K}}_i(cl(H))={\cal{K}}_i(H),\ {\cal{K}}_i(H')={\cal{K}}_i(H)$ respectively then it is inherited to ${\cal{K}}$.

Convex: Let $I$ be a closed interval and ${\cal{K}}(H)\in I,\ L\subset I,H\cup L\in Dom({\cal{K}})\cap_1^{\infty}Dom({\cal{K}}_i)$. Then ${\cal{K}}_i(H)\in I_i$ where $I_i$ is the closed convex hull of ${\cal{K}}_i(H)$ and $I$. This implies that ${\cal{K}}_i(H\cup L)\in I_i$. But ${\cal{K}}_i(H)\to{\cal{K}}(H)$ therefore $\cap_1^{\infty}I_i=I$ which yields that ${\cal{K}}(H\cup L)\in I$. 

Translation-invariant, reflection-invariant, homogeneous: If $\forall i\ {\cal{K}}_i(H+x)={\cal{K}}_i(H)+x,\ T_s({\cal{K}}_i(H))={\cal{K}}_i(T_s(H)),\ \alpha{\cal{K}}_i(H)={\cal{K}}_i(\alpha H)$ respectively then it is inherited to ${\cal{K}}$ because $f(y)=y+x,\ g(y)=T_s(y),\ h(y)=\alpha y$ are continuous.
\end{proof}

\begin{lem}\label{lucinh}Let $H,H_i\in \cap_1^{\infty}Dom({\cal{K}}_j)\cap Dom({\cal{K}})$ and $\forall n\in\mathbb{N}\ \lim\limits_{i\to\infty}{\cal{K}}_n(H_i)={\cal{K}}_n(H)$. If ${\cal{K}}_i\to{\cal{K}}$ uniformly then $\lim\limits_{i\to\infty}{\cal{K}}(H_i)={\cal{K}}(H)$.
\end{lem}
\begin{proof}Let $\epsilon>0$. Then there is $N\in\mathbb{N}$ such that $n\geq N$ implies that $|{\cal{K}}_n(L)-{\cal{K}}(L)|<\frac{\epsilon}{3}$ for all $L\in \cap_1^{\infty}Dom({\cal{K}}_j)\cap Dom({\cal{K}})$. We know that $\lim\limits_{i\to\infty}{\cal{K}}_N(H_i)={\cal{K}}_N(H)$ hence there is $I\in\mathbb{N}$ such that $i\geq I$ implies that $|{\cal{K}}_N(H_i)-{\cal{K}}_N(H)|<\frac{\epsilon}{3}$. Now we get that
\[|{\cal{K}}(H_i)-{\cal{K}}(H)|\leq|{\cal{K}}(H_i)-{\cal{K}}_N(H_i)|+|{\cal{K}}_N(H_i)-{\cal{K}}_N(H)|+|{\cal{K}}_N(H)-{\cal{K}}(H)|<\epsilon.\qedhere\]
\end{proof}

\begin{thm}\label{tucipr}If ${\cal{K}}_n\to{\cal{K}}$ uniformly and $\forall n\ {\cal{K}}_n$ is Cantor-continuous, slice-continuous, Hausdorff-continuous respectively then so is ${\cal{K}}$.
\end{thm}
\begin{proof}Apply \ref{lucinh} in the different cases as follows.

Cantor-continuous: Let $H_{i+1}\subset H_i,\ H=\cap_1^{\infty}H_i$.

Slice-continuous: Let $x_n\to x,\ H_i=H^{x_i-}, H=H^{x-}$. 

Hausdorff-continuous: Let $H_i,H$ be compact and $H_i\to H$ in the Hausdorff-metric.
\end{proof}

\begin{cor}The convergence of $Avg(S(H,\frac{1}{n}))$ to $LAvg(H)$ is not uniform.
\end{cor}
\begin{proof}By \ref{tavgsdhc} and \ref{plavgnhc} $Avg(S(H,\frac{1}{n}))$ is Hausdorff-continuous while $LAvg(H)$ is not. Hence by \ref{tucipr} the convergence cannot be uniform.
\end{proof}

\section{Some open problems}

\begin{prb}Find a mean that is slice-continuous but fails to be bi-slice-continuous.
\end{prb}

\begin{prb}Let ${\cal{K}}=Avg^1, H\in Dom(Avg^1)$. Show that $f(x)=\overline{d}{\cal{K}}_H(x)$ is continuous at almost every point of $H^{'{\cal{K}}}$.
\end{prb}

\begin{prb}Let ${\cal{K}}=Avg^1, H\in Dom(Avg^1)$. Show that for almost every point of $H^{'{\cal{K}}}$ $d{\cal{K}}_H(x)$ exists.
\end{prb}


{\footnotesize

\smallskip
\noindent
Dennis G\'abor College, Hungary 1119 Budapest Fej\'er Lip\'ot u. 70.

\noindent 
E-mail: losonczi@gdf.hu, alosonczi1@gmail.com\\
}

\end{document}